\documentclass[12pt,reqno,oneside]{amsproc}
\title[{Hausdorff measure of nodal sets}]{Upper bounds of nodal sets for Gevrey regular
	\\ parabolic equations}

\author[G.~Camliyurt]{Guher Camliyurt}
\address{Department of Mathematics\\
Virginia Polytechnic Institute and State University\\
Blacksburg, VA 24061
}
\email{gcamliyurt@vt.edu}

\author[I.~Kukavica]{Igor Kukavica}
\address{Department of Mathematics\\
	University of Southern California\\
	Los Angeles, CA 90089}
\email{kukavica@usc.edu}

\author[L.~Li]{Linfeng Li}
\address{Department of Mathematics\\
	University of California Los Angeles\\
	Los Angeles, CA 90095}
\email{lli265@math.ucla.edu}

\usepackage{enumitem}

\usepackage{fancyhdr}
\usepackage{comment}

  \chardef\forshowkeys=0
  \chardef\refcheck=0
  \chardef\showllabel=0
  \chardef\sketches=0

\ifnum\forshowkeys=1
  
  \usepackage[notcite,color]{showkeys}
\fi

\usepackage[margin=1.1in]{geometry}
\usepackage{amsmath, amsthm, amssymb}
\usepackage{times}
\usepackage{graphicx}
\usepackage[usenames,dvipsnames,svgnames,table]{xcolor}
\usepackage{marginnote}
\usepackage[unicode,breaklinks=true,colorlinks=true,linkcolor=black,urlcolor=black,citecolor=black]{hyperref}

\usepackage{tikz}

\ifnum\refcheck=1
  \usepackage{refcheck}
\fi

\begin{document}
\def\linfeng#1{\textcolor{red}{#1}}
\def\DD{\text{D}}
\def\PP{\mathcal{P}}
\def\dPP{\dot{\mathcal{P}}}
\def\LL{\mathcal L}
\def\qd{q_{\textrm D}^{}}
\def\AA{C_2}
\def\BB{(C_1+C_2)}
\def\pp{p}
\def\qq{q}
\def\MM{M}
\def\KK{K}
\def\CC{C}
\def\uu{\tilde u }
\def\vv{\tilde v }
\def\ww{\tilde w }
\def\ggg{u}
\def\UU{\tilde U}
\def\eps{\epsilon}
\def\QQ{{\bar Q}}

\def\PP{P}
\def\RR{\mathbb R}
\def\TT{\mathbb T}
\def\ZZ{\mathbb Z}
\def\erf{\mathrm{Erf}}
\def\red#1{\textcolor{red}{#1}}
\def\blue#1{\textcolor{blue}{#1}}
\def\mgt#1{\textcolor{magenta}{#1}}

\def\les{\lesssim}
\def\ges{\gtrsim}

\renewcommand*{\Re}{\ensuremath{\mathrm{{\mathbb R}e\,}}}
\renewcommand*{\Im}{\ensuremath{\mathrm{{\mathbb I}m\,}}}

\ifnum\showllabel=1
 \def\llabel#1{\marginnote{\color{lightgray}\rm\small(#1)}[-0.0cm]\notag}
\else
 \def\llabel#1{\notag}
\fi

\newcommand{\norm}[1]{\left\|#1\right\|}
\newcommand{\nnorm}[1]{\lVert #1\rVert}
\newcommand{\abs}[1]{\left|#1\right|}
\newcommand{\NORM}[1]{|\!|\!| #1|\!|\!|}

\newtheorem{Theorem}{Theorem}[section]
\newtheorem{Corollary}[Theorem]{Corollary}
\newtheorem{Definition}[Theorem]{Definition}
\newtheorem{Proposition}[Theorem]{Proposition}
\newtheorem{Lemma}[Theorem]{Lemma}
\newtheorem{Remark}[Theorem]{Remark}

\def\theequation{\thesection.\arabic{equation}}
\numberwithin{equation}{section}

\definecolor{myblue}{rgb}{.8, .8, 1}

\newlength\mytemplen
\newsavebox\mytempbox

\makeatletter
\newcommand\mybluebox{%
    \@ifnextchar[
       {\@mybluebox}%
       {\@mybluebox[0pt]}}

\def\@mybluebox[#1]{%
    \@ifnextchar[
       {\@@mybluebox[#1]}%
       {\@@mybluebox[#1][0pt]}}

\def\@@mybluebox[#1][#2]#3{
    \sbox\mytempbox{#3}%
    \mytemplen\ht\mytempbox
    \advance\mytemplen #1\relax
    \ht\mytempbox\mytemplen
    \mytemplen\dp\mytempbox
    \advance\mytemplen #2\relax
    \dp\mytempbox\mytemplen
    \colorbox{myblue}{\hspace{1em}\usebox{\mytempbox}\hspace{1em}}}

\makeatother

\def\biglinem{\vskip0.5truecm\par==========================\par\vskip0.5truecm}
\def\inon#1{\hbox{\ \ \ \ \ }\hbox{#1}}                
\def\and{\text{\indeq and\indeq}}
\def\onon#1{\text{~~on~$#1$}}
\def\inin#1{\text{~~in~$#1$}}

\def\startnewsection#1#2{\section{#1}\label{#2}\setcounter{equation}{0}}   
\def\sgn{\mathop{\rm sgn\,}\nolimits}    
\def\Tr{\mathop{\rm Tr}\nolimits}    
\def\div{\mathop{\rm div}\nolimits}
\def\curl{\mathop{\rm curl}\nolimits}
\def\dist{\mathop{\rm dist}\nolimits}
\def\card{\mathop{\rm card}\nolimits}  
\def\sp{\mathop{\rm sp}\nolimits}  
\def\supp{\mathop{\rm supp}\nolimits}
\def\indeq{\quad{}}           
\def\colr{\color{red}}
\def\colb{\color{black}}
\def\coly{\color{lightgray}}

\definecolor{colorgggg}{rgb}{0.1,0.5,0.3}
\definecolor{colorllll}{rgb}{0.0,0.7,0.0}
\definecolor{colorhhhh}{rgb}{0.3,0.75,0.4}
\definecolor{colorpppp}{rgb}{0.7,0.0,0.2}
\definecolor{coloroooo}{rgb}{0.45,0.0,0.0}
\definecolor{colorqqqq}{rgb}{0.1,0.7,0}
\def\colg{\color{colorgggg}}
\def\collg{\color{colorllll}}
\def\cole{\color{coloroooo}}
\def\cole{\color{black}}
\def\coleo{\color{colorpppp}}
\def\colu{\color{blue}}
\def\colc{\color{colorhhhh}}
\def\colW{\colb}   
\definecolor{coloraaaa}{rgb}{0.6,0.6,0.6}
\def\colw{\color{coloraaaa}}

\def\comma{ {\rm ,\qquad{}} }            
\def\commaone{ {\rm ,\quad{}} }          
\def\nts#1{{\color{red}\hbox{\bf ~#1~}}} 
\def\ntsf#1{\footnote{\color{colorgggg}\hbox{#1}}} 
\def\blackdot{{\color{red}{\hskip-.0truecm\rule[-1mm]{4mm}{4mm}\hskip.2truecm}}\hskip-.3truecm}
\def\bluedot{{\color{blue}{\hskip-.0truecm\rule[-1mm]{4mm}{4mm}\hskip.2truecm}}\hskip-.3truecm}
\def\purpledot{{\color{colorpppp}{\hskip-.0truecm\rule[-1mm]{4mm}{4mm}\hskip.2truecm}}\hskip-.3truecm}
\def\greendot{{\color{colorgggg}{\hskip-.0truecm\rule[-1mm]{4mm}{4mm}\hskip.2truecm}}\hskip-.3truecm}
\def\cyandot{{\color{cyan}{\hskip-.0truecm\rule[-1mm]{4mm}{4mm}\hskip.2truecm}}\hskip-.3truecm}
\def\reddot{{\color{red}{\hskip-.0truecm\rule[-1mm]{4mm}{4mm}\hskip.2truecm}}\hskip-.3truecm}

\def\tdot{{\color{green}{\hskip-.0truecm\rule[-.5mm]{2mm}{4mm}\hskip.2truecm}}\hskip-.2truecm}
\def\gdot{\greendot}
\def\bdot{\bluedot}
\def\ydot{\cyandot}
\def\rdot{\cyandot}

\def\fractext#1#2{{#1}/{#2}}

\def\quinn#1{\text{{\textcolor{colorqqqq}{#1}}}}
\def\igor#1{\footnote{\text{{\textcolor{colorqqqq}{#1}}}}}

%
%
%
%

\begin{abstract}
We
consider the size of the nodal set of the solution of the second order
parabolic-type equation with Gevrey regular coefficients. We provide an upper bound as a function of time. The dependence agrees with a sharp
upper bound when the coefficients are analytic.
\hfill \today
\end{abstract}

\maketitle

\setcounter{tocdepth}{2} 

\tableofcontents

\colb
\startnewsection{Introduction}{sec01}
We consider the size of the nodal (zero) set to the solution of the parabolic equation
\begin{align}
	u_t -\Delta u
	=
	w \cdot \nabla u
	+v u,
	\label{eq0}
\end{align}
in $\TT^d\times (0,T)$, where $d\geq 1$, with given $v$ and $w$ and the initial data $u(\cdot,0)=u_0\in H^1 (\TT^d)$. Under the assumption that $v$ and $w$ are $\frac{1}{\beta}$-Gevrey regular, we provide an upper bound on the $(d-1)$-dimensional Hausdorff measure of the nodal set. Our main result, Theorem~\ref{T01}, shows that for any $\beta\in (0,1]$ and $t\in (0,t_0]$, we have
\begin{align}
	\mathcal{H}^{d-1}
	(\{x\in \TT^d: u(x,t)=0\} )
	\leq
	C	\left(\frac{1}{t} \right)^{\fractext{1}{\beta}} \log^{2(\fractext{1}{\beta}-1)} \left(\frac{1}{t} \right)
	,
	\label{eq1}
\end{align}
for some constant $C>0$ depending on the initial condition of $u_0$,
and the regularity of $v$ and~$w$, assuming $t_0\in(0,1/2]$. 

In \cite{Y}, Yau conjectured that the Laplacian eigenfunctions on a compact smooth Riemannian manifold $M$ satisfy 
\begin{align}
	C_1 \lambda^{1/2}
	\leq
	\mathcal{H}^{d-1}
	(\{x\in M: u(x) = 0\})
	\leq
	C_2 \lambda^{1/2}
	,
   \label{EQ06}
\end{align}
where $C_1, C_2>0$ are some constants. For the eigenfunctions to the Laplacian equation $\Delta u = - \lambda u$, Donnelly and Fefferman established in \cite{DF1, DF2, DF3} the lower and upper bounds on Hausdorff measure of nodal sets on compact Riemannian manifolds, establishing Yau's conjecture for real analytic manifolds with or without boundary. In~\cite{K1}, the author obtained a sharp upper bound for arbitrary even order real analytic elliptic operator. More recently, Hezari provided in \cite{He} an upper bound for the size of nodal sets of eigenfunctions when the Riemannian manifold has Gevrey or quasianalytic regularity, using a previous result in \cite{IK1} on the size of the zero set for Gevrey functions. Note that the paper \cite{IK1} provides an upper bound for the zero set of a 1D parabolic equation with Gevrey coefficients. Motivated by \eqref{EQ06}, Nadirashvili conjectured in \cite{N} that
\begin{align*}
	\mathcal{H}^{d-1}
	(\{x\in B(0,1): u(x) = 0\}) 
	\geq 
	c,
\end{align*}
where $u$ is a harmonic function in the unit ball $B(0,1)$ and $c>0$ depends only on the dimension. 

In recent years, Logunov and Malinnikova achieved significant progress towards Yau's conjecture for smooth manifolds~\cite{Lo1,Lo2,LoM1}. In~\cite{Lo2}, Logunov proved Nadirashvili's conjecture and showed the lower bound in Yau's conjecture for compact $C^\infty$-smooth Riemannian manifold without boundary. In \cite{LoM1}, the authors established further estimates on the nodal sets of Laplace eigenfunctions; their result refined the upper bound in the Donnelly-Fefferman estimate when $d=2$ and yielded a lower bound when $d=3$. We refer the reader to the review of Yau's conjecture in \cite{LoM2} and for further results and references. Also, see \cite{D,Do,HL,HJ,HS,KL,L,LZ,TY} for other results on the size of the nodal sets of solutions of partial differential equations.



The study of nodal sets is closely related to the unique continuation for solutions of the parabolic equation.
However, functions of Gevrey class in general may not satisfy the strong unique continuation property. 
Namely, the function may have infinite order of vanishing even if it is not identically equal to zero.
In \cite{IK1}, the authors gave a quantitative estimate of unique continuation for one-dimensional parabolic equation with Gevrey coefficients. 
As an application, they provided an upper bound on the number of zeros with a polynomial dependence on the coefficients.
The same authors considered in \cite{IK2} a higher order elliptic equation with Gevrey coefficients. They obtained an upper bound on the Hausdorff length of the nodal sets of the solution in a domain in $\mathbb{R}^2$ with a polynomial dependence on the coefficients.
In~\cite{K3}, the author provided an upper bound of the nodal set to the parabolic equation $\partial_t u + \mathcal{A} u=0$, where $\mathcal{A}$ is strongly elliptic and real-analytic.

The goal of this paper is to provide an explicit upper bound of the Hausdorff measure for the nodal set of the solution to the parabolic equation \eqref{eq0} when $v$ and $w$ are Gevrey regular.
Setting the Gevrey exponent $1/\beta$ equal to~$1$ in \eqref{eq1}, we provide the same power in $1/t$ as in~\cite{K3}.
It is worth mentioning that the power in $1/t$ in the bound of
\cite{K3} is sharp
as it implies the upper bounds in~\cite{Do,DF1}.
Our result is obtained by using the quantitative estimate of unique continuation for Gevrey parabolic equations (cf.~Lemmas~\ref{Lku2}--\ref{Lku3}) and a Gevrey-type energy estimate lemma.
The proof of the main result benefits from the ideas used in~\cite{IK1}.

This paper is organized as follows.
In Section~\ref{sec02}, we introduce the parabolic equation and the assumptions, followed by the main result, Theorem~\ref{T01}, an upper bound on the size of the nodal set. In Section~\ref{sec03}, we recall and prove some preliminary lemmas and provide
the proof of the main result.

\startnewsection{G·evrey regular parabolic problem and the main result}{sec02}
We consider the solution to the parabolic equation
\begin{align}
	&u_t -\Delta u
	=
	w \cdot \nabla u
	+
	v u
	\comma
	(x,t)\in \Omega \times I,
	\label{para01}
	\\
	&
	u(x,0)=u_0(x)
	\comma
	x\in \Omega,
	\label{para0}
\end{align}
where $\Omega=\TT^d$ is the $d$-dimensional torus with side length $2\pi$ and $I=(0,t_0]$ is a given time interval, where $t_0 \in (0,1/2]$. We require that $u$, $v$, and $w$ are $2\pi$-periodic.
We assume that $u_0 \in H^1 (\TT^d)$ is not identically zero.
Suppose that there exist constants $M_0, M_1\geq1$ such that
\begin{align}
	\sup_{t\in I} 
	\Vert v(t) \Vert_{L^\infty (\TT^d)}
	\leq M_0
	\label{para03}
\end{align}
and
\begin{align}
	\sup_{t\in I} 
(	\Vert w(t) \Vert_{L^\infty (\TT^d) }
	+
	\Vert \nabla w(t) \Vert_{L^\infty (\TT^d) })
	\leq M_1.
	\label{para02}
\end{align}
For any $t\in I$, denote by
\begin{align}
	q_{\DD}(t)
	=
	\frac{\int_\Omega |\nabla u(x,t)|^2\,dx}{\int_\Omega u(x,t)^2 \,dx}
	\label{para04}
\end{align}
the Dirichlet quotient of $u$, and
\begin{align}
	q_0=\sup_{t\in I}
	q_{\DD}(t)
	\label{para05}
\end{align}
the uniform upper bound.
We identify functions $u\in L^2 (\TT^d)$ with their Fourier series
\begin{align}
	u(x)=\sum_{j\in \ZZ^d} u_j e^{ijx}
	\comma
	u_j\in \mathbb{C}^d
	\comma
	u_{-j}=\bar{u}_j.
   \llabel{EQ01}
\end{align}
For $u,v\in L^2 (\TT^d)$, we define  the scalar product
\begin{align}
	(u,v)= (2\pi)^d
	\sum_{j\in \ZZ^d}
	u_j 
	\bar{v}_j
   \llabel{EQ02}
\end{align}
and the norm $\Vert u\Vert_{L^2(\TT^d)}^2 = (u,u)$.
Let $A=-\Delta$. For fixed $t,\beta>0$, we introduce
the Gevrey class
\begin{align}
	D(e^{t A^{\beta/2}})
	=
	\biggl\{u\in L^2(\TT^d):
	\Vert e^{t A^{\beta/2}} u \Vert_{L^2}^2
	=
	(2\pi)^d 
	\sum_{j\in \ZZ^d}
	e^{2t |j|^{\beta}}
	|u_j|^2<\infty
	\biggr\}.
   \llabel{EQ03}
\end{align}
From here on, we fix $\beta \in (0,1]$ and assume that there exist constants $\delta, K_v, K_w>0$ such that
\begin{align}
	\sup_{t\in I}
	\Vert (A^d + I) e^{\delta A^{\beta/2}} v (t)\Vert_{L^2 (\TT^d)}
	\leq 
	K_v
	\label{EQ52}
\end{align}
and
\begin{align}
	\sup_{t\in I}
	\Vert (A^d + I) e^{\delta A^{\beta/2}} w(t)\Vert_{L^2 (\TT^d)}
	\leq 
	K_w.
	\label{EQ72}
\end{align}

Let $t\in I=(0,t_0]$. We are interested in the size of nodal set $\{x\in \TT^d: u(x,t) = 0\}$ of the nontrivial solution to \eqref{para01} and \eqref{para0} with Gevrey regular $v$ and~$w$.
Denote by
  \begin{equation}
   N(t)=\{x\in \TT^d: u(x,t) = 0\}
   \llabel{EQ04}
  \end{equation}
the nodal set, for which we provide the upper bound in the theorem stated next.  
Throughout this paper, we assume that $t_0\leq 1/CM_1^2$, where $C\geq2$ is the constant as in \cite[Lemma~2.5]{K2}.
We can do so without loss of generality since
if $t_0$ is larger, then we 
obtain \eqref{EQ26} below
for $t\leq 1/C M_1^2$ and
\eqref{EQ26} with $C$ on the right-hand side otherwise.

The following is the main result of the paper.
\cole
\begin{Theorem}
	\label{T01}
Suppose that $u(x,t)$ is the solution to \eqref{para01} and \eqref{para0}
on $\TT^d\times I$ with  $u_0 \in H^1 (\TT^d)$
not identically zero.
Assume that $v$ and $w$ satisfy \eqref{para03}--\eqref{para02} and \eqref{EQ52}--\eqref{EQ72}.
Then, for each $t\in I$, we have
	\begin{align}
		\mathcal{H}^{d-1}
		(N(t) )
		&\leq
		C	\left(\frac{1}{t} \right)^{\frac{1}{\beta}} \log^{2 (\frac{1}{\beta}-1)} \left(\frac{1}{t} \right)
		,
		\label{EQ26}
	\end{align}
where $C>0$ is a constant depending on $M_0$, $M_1$, $K_v$ and~$K_w$.
\end{Theorem}
\colb

The theorem is proven in the next section.


Here we comment on the boundedness of the Dirichlet quotient~$q_\DD$ in \eqref{para04}.
As in \cite[p.~463]{FS}, we have
  \begin{equation*}
   \frac12
   \frac{d}{dt}
   q_{\DD}
   +
   \left\Vert
    (-\Delta-q_{\DD}) \frac{u}{\Vert u\Vert_{L^2}}
   \right\Vert^2
   =
   \frac{1}{\Vert u\Vert^2}
   \int_{\Omega}
   (w\cdot \nabla u + v u)
   (-\Delta-q_{\DD}) u
  \end{equation*}
which implies by Cauchy-Schwarz inequality
  \begin{equation*}
   \frac{d}{dt}
   q_{\DD}
   \leq
   C (
      \Vert w\Vert_{L^\infty}^2
      q_{\DD}
      +\Vert v\Vert_{L^\infty}^2)
   ,   
  \end{equation*}
and this gives an explicit bound on $q_\DD(t)$ for
$t \in (0,t_0]$
depending on~$q_\DD(t_0)$.

\startnewsection{Proof of the main theorem}{sec03}
First we recall the following lemma from~\cite{K2}.

\cole
\begin{Lemma}[\cite{K2}]
	\label{Lku2}
	Let $u$ be the solution to \eqref{para01} and \eqref{para0} with $v$ and $w$ satisfying \eqref{para03} and~\eqref{para02}. Then for every $r\in (0,1]$ and $p\in \TT^d$, we have
	\begin{align}
		\Vert u( t)\Vert_{L^2(\TT^d)}^2
		\leq
		\exp{(
			P(t^{-1}, r^{-1}, q_0, M_0, M_1)	
			)}
		\Vert u(t)\Vert^2_{L^2 (B_{r}(p))}
   \llabel{EQ18}
	\end{align}
for all $t\in (0,t_0]$, where $P$ is a nonnegative polynomial and $q_0$ is as in \eqref{para05}.
\end{Lemma}
\colb

Since the dependence of $P$ on $r$ and $t$, encoded in \cite[Lemma~2.6]{K2},
is important, we recall the following result.

\cole
\begin{Lemma}[\cite{K2}]
	\label{Lku3}
	Let $u$ be the solution to \eqref{para01} and \eqref{para0} with $v$ and $w$ satisfying \eqref{para03} and~\eqref{para02}. Then for every $t\in (0,t_0]$ and every $\mu_0 \in (0,1]$, we have
	\begin{align}
		\Vert u(t)\Vert_{L^2 (\TT^d)}^2
		\leq
		C e^{\mu_0^2/4\mu^2}
		\Vert u(t)\Vert_{L^2 (B_{\mu_0} )}^2,
		\label{EQ83}
	\end{align}
where $\mu>0$ is determined in the following way: Denote
\begin{align}
\begin{split}
	\beta (\mu)
	&=C\left( tq_0+M_1^2 +M_1^2 t
	+M_0^2 t^2 +\frac{1}{t}
	+\frac{1}{t^{1/2}} 
	\right)
	\log \frac{t}{\mu^2}
	\\&\indeq	\indeq
	+
	C\left(
	M_1 t^{1/2}+M_1^2 t +M_0 t +\frac{1}{t}+\frac{1}{t^{1/2}}
	\right);
\end{split}
   \llabel{EQ05}
\end{align}
then $\mu$ is obtained by requiring
\begin{align}
	\frac{1}{\mu^2}
	\geq 
	\frac{C}{\mu_0^2} \log \frac{1}{\mu}
	+
	\frac{C ( \beta (\mu)+1)}{\mu_0^2}
	\label{a01}
\end{align}	
and
\begin{align}
	0<\mu <\min \left\{
	\sqrt{\frac{t}{2}}, \mu_0
	\right\}.
	\label{a02}
\end{align}
\end{Lemma}	
\colb	

Note that for any $\mu_0 \in (0,1/2]$, we may take
\begin{align}
	\mu=
	\frac{\mu_0 \sqrt{\frac{t}{2}}}{K\log^{\frac{1}{2}} \frac{1}{\mu_0}},
   \llabel{EQ19}
\end{align}
where $K>0$ is a sufficiently large constant depending on $q_0$, $M_0$ and $M_1$, with the assumptions \eqref{a01} and \eqref{a02} on $\mu$ satisfied. Thus, from \eqref{EQ83}, it follows that there exists a constant $C>0$ such that
\begin{align}
	\Vert u(t)\Vert_{L^2 (\TT^d)}^2
	\leq
	C^
	{K^2 \log (\frac{1}{\mu_0}) \frac{1}{t}}
	\Vert u(t)\Vert_{L^2 (B_{\mu_0})}^2,
	\label{EQ84}
\end{align}
for any $\mu_0\in (0,1/2]$ and $t\in (0,t_0]$.
%


The following lemma provides a Gevrey estimate for the parabolic
equation~\eqref{para01}--\eqref{para0}, combined with the backward
uniqueness for \eqref{para01}--\eqref{para0}.

\cole
\begin{Lemma}
	\label{Lgevrey}
Denote $A=-\Delta$.
	Let $u(x,t)$ be the solution to \eqref{para01} with $u_0\in L^2 (\TT^d)$. Suppose that $v$ and $w$ satisfy \eqref{para03}--\eqref{para02} and \eqref{EQ52}--\eqref{EQ72}. Then we have
	\begin{align}
		\Vert e^{\delta A^{\beta/2}} u(t)		
		\Vert_{L^2 (\TT^d)}
		\leq
		C^{(
		\delta ^{2/(2-\beta)} t^{-\beta/(2-\beta)} + t(K_w^2 +K_v  +M_1+M_0 +q_0) )}
		\Vert u(t)\Vert_{L^2 (\TT^d)}
   \llabel{EQ20}
	\end{align}
for any $t\in (0,t_0]$, where $C>0$ is a constant.
\end{Lemma}
\colb

\begin{proof}
Fix $t\in (0,t_0]$.
	Let $\alpha\in (0, \delta/t]$.
	By taking the $L^2$-inner product of \eqref{para01} with $e^{\alpha t A^{\beta/2}} u$, we obtain
	\begin{align}
		\begin{split}
			&
		\frac{1}{2}
		\frac{d}{dt}
		\Vert e^{\alpha t A^{\beta/2}} u(t)\Vert_{L^2 (\TT^d)}^2
		+
		\Vert A^{\frac{1}{2}} e^{\alpha t A^{\beta/2}} u\Vert_{L^2 (\TT^d)}^2
		\\&=
		(e^{\alpha t A^{\beta/2}} (w\cdot \nabla u), e^{\alpha t A^{\beta/2}} u)
		+
		\alpha	
		\Vert A^{\beta/4}
		e^{\alpha t A^{\beta/2}} u
		\Vert_{L^2}^2
		+
		(e^{\alpha t A^{\beta/2}} (vu), e^{\alpha t A^{\beta/2}} u)
		\\&
		=
		I_1+I_2+I_3.
		\label{EQ53}
		\end{split}
	\end{align}
The term $I_1$ is estimated as in \cite[Lemma~2.1]{FT}, leading to
	\begin{align}
		\begin{split}
				I_1
		&\leq
		C \Vert e^{\alpha t A^{\beta/2}} (A^{d}+I) w \Vert_{L^2 (\TT^d)}
		\Vert e^{\alpha t A^{\beta/2}} A^{\frac{1}{2}} u \Vert_{L^2 (\TT^d)}
		\Vert e^{\alpha t A^{\beta/2}} u \Vert_{L^2 (\TT^d)}
		\\&
		\leq
		\frac{1}{4}
		\Vert e^{\alpha t A^{\beta/2}} A^{\frac{1}{2}} u \Vert_{L^2 (\TT^d)}^2
		+
		C K_w^2 \phi (t)
		,
		\label{EQ80}
		\end{split}
	\end{align}
for some constant $C>0$,
where we denoted $\phi(t)= \Vert e^{\alpha t A^{\beta/2}} u(t)\Vert_{L^2 (\TT^d)}^2$. 
For the term $I_2$, we appeal to Young's inequality to get
	\begin{align}
		\begin{split}
		I_2
		&=
		\alpha
		\sum_{j\in \ZZ^d}
		|j|^\beta
		e^{2\alpha t |j|^{\beta}} |u_j|^2
		\leq
		\frac{1}{4}
		\sum_{j\in \ZZ^d}
		|j|^2 e^{2\alpha t |j|^\beta} |u_j|^2
		+
		C \alpha^{2/(2-\beta)}
		\sum_{j\in \ZZ^d}
		e^{2\alpha t |j|^\beta} |u_j|^2
		\\&
		=
		\frac{1}{4}
		\Vert A^{\frac{1}{2}} e^{\alpha t A^{\beta/2}} u\Vert_{L^2 (\TT^d)}^2
		+
		C \alpha^{2/(2-\beta)}
		\phi (t),
				\label{EQ54}
	\end{split}
	\end{align}
for some constant $C>0$.
For any $\tau\geq 0$, we follow the proof of \cite[Lemma~2.1]{FT} to obtain
\begin{align}
	|(e^{\tau A^{\beta/2}} (uv), e^{\tau A^{\beta/2}} u)|
	\leq
	C
	\Vert (A^d+I) e^{\tau A^{\beta/2}} v \Vert_{L^2 (\TT^d)}
	\Vert e^{\tau A^{\beta/2}} u\Vert_{L^2 (\TT^d)}
	\Vert e^{\tau A^{\beta/2}} u\Vert_{L^2 (\TT^d)}.
	\label{EQ51}
\end{align}
	For the term $I_3$, we use \eqref{EQ52} and \eqref{EQ51}, obtaining
	\begin{align}
	I_3
	\leq 	
	C
	\Vert (A^d+I) e^{\delta A^{\beta/2}} v \Vert_{L^2 (\TT^d)}
	\Vert e^{\alpha t A^{\beta/2}} u\Vert_{L^2 (\TT^d)}^2
	\leq 
	CK_v \phi (t),
	\label{EQ55}
	\end{align}
for some constant $C>0$, where we used $\alpha t\leq \delta$ in the first inequality.
From \eqref{EQ53}--\eqref{EQ54} and \eqref{EQ55} it follows that
\begin{align}
	\phi' (t)
	\leq 
	C(\alpha^{2/(2-\beta)}+ K_w^2+K_v)
	\phi (t),
   \llabel{EQ21}
\end{align}
from where
\begin{align}
	\phi (t) 
	\leq 
	\exp (C t(\alpha^{2/(2-\beta)}+K_w^2+K_v))  \phi (0).
	\llabel{EQ56}
\end{align}
Letting $\alpha=\delta/t$, we get
\begin{align}
	\Vert e^{\delta A^{\beta/2}} u(t)\Vert_{L^2 (\TT^d)}
	\leq
		\exp (C t (\delta^{2/(2-\beta)}  t^{-2/(2-\beta)}
		+
		K_w^2
		+
		K_v)) \Vert u_0\Vert_{L^2 (\TT^d)}.
		\label{EQ70}
\end{align}

Taking the inner product of \eqref{para01} with $u$ and using $\Vert \nabla u\Vert^2_{L^2 (\Omega)}= q_{\DD}(t) \Vert u\Vert_{L^2(\Omega)}^2$, where $q_{\DD}(t)$ is defined in \eqref{para04},
we obtain
\begin{align}
	\frac{1}{2}
	\frac{d}{dt}
	\Vert u\Vert^2_{L^2 (\TT^d)}
	+q_{\DD}(t) \Vert u(t)\Vert_{L^2 (\TT^d)}^2
	=
	-\frac{1}{2}
	\int_{\TT^d} (\div w)
	u^2 \,dx
	+
	\int_{\TT^d}
	v u^2 \,dx.
   \llabel{EQ22}
\end{align}
Using H\"older's inequality, it follows that
\begin{align}
	\begin{split}
		\frac{1}{2}
	\frac{d}{dt}
	\Vert u\Vert^2_{L^2 (\TT^d)}
	&\geq 
	- (
	\Vert \div w\Vert_{L^\infty (\TT^d)}
	+
	\Vert v(t)\Vert_{L^\infty (\TT^d)} + q_{\DD}(t))
	\Vert u(t)\Vert_{L^2 (\TT^d)}^2
	\\&
	\geq 
	-(M_1+M_0 + q_0)
	\Vert u(t)\Vert_{L^2 (\TT^d)}^2.	
	\end{split}
   \llabel{EQ23}
\end{align}
Applying Gronwall's argument, we arrive at
\begin{align}
	\Vert u(t)\Vert_{L^2 (\TT^d)}^2
	\geq 
	\exp (-2t(M_1+M_0 + q_0) )
	\Vert u_0\Vert_{L^2 (\TT^d)}^2
	,
	\label{EQ71}
\end{align}
and the proof is concluded by combining \eqref{EQ70} and~\eqref{EQ71}.
\end{proof}

Next we provide an upper bound estimate of Hausdorff measure of a set $N\subset\RR^d$.
\cole
\begin{Lemma}
\label{L03}
Let $r>0$. Suppose that $N\subset B_{2r}$ is a subset of a countable union of $(d-1)$-dimensional $C^1$ graphs.
Let $T=\{p_{\pm 1}, \ldots, p_{\pm d}\}$, where $p_j \in B_{r/10d} (r e_j)$ with the standard basis $e_j$ in $\mathbb{R}^d$, for $j = \pm 1, \ldots, \pm d$.
If $N$ has at most $n$ intersections with any line through any $p\in T$, then
	\begin{align}
		\mathcal{H}^{d-1} (N)
		\leq
		C n r^{d-1}
		,
		\label{EQ39}
	\end{align}
where $C>0$ is a constant independent of $N$, $n$ and~$r$.
\end{Lemma}
\colb

\begin{proof}[Proof of Lemma~\ref{L03}]
Without loss of generality, we may assume that $r=1$ and establish
	\begin{align}
	\mathcal{H}^{d-1} (N)
	\leq
	C n
	,
	\label{EQ47}
\end{align}
for some constant $C>0$.
Fix $q\in N$. 
We claim that there are at most $n$ pairwise non-parallel unit normal vectors to $N$ at~$q$. Suppose, for the sake of contradiction, that there are $n+1$ pairwise non-parallel unit normal vectors at~$q$.	
There exists $\kappa>0$ and a line $l$ through a point $p\in T$ such that
\begin{align*}
	\card (N \cap B_\kappa (q) \cap l) \geq n+1.
\end{align*}
The existence of such a line is obtained by perturbing the line connecting $p$ and $q$ and applying the implicit function theorem.
This contradicts the assumption that every line through $p\in T$ intersects $N$ in at most $n$ points.
Let $\nu_q= (\mu_1, \ldots, \mu_d)$ be a unit normal vector to $N$ at $q$ and $|\mu_k| = \max \{|\mu_1|, \ldots, |\mu_d|\}$ for some $k\in \{1,2,\ldots, d\}$.
We claim that there exists some $p \in T$ such that
\begin{align}
	\begin{split}
		\frac{|(q -p) \cdot \nu_q |}{|q - p|}	
		\geq
		\frac{1}{4\sqrt{d}}
		\label{EQ44}
	\end{split}
\end{align}
and
\begin{align}
	|q-p| \geq 1/10d.
	\label{EQ45}
\end{align}
Let $p_j=  e_j+ w_j /10 d$, where $w_j \in B_1$, for $j = \pm 1, \ldots, \pm d$.
If $q\in N \setminus ( B_{1/5d} (e_{k}) \cup  B_{1/5d} (e_{-k}))$, we get
\begin{align}
	\begin{split}
		\frac{|(q -p_k) \cdot \nu_q |}{|q - p_k|}	
		+
		\frac{|(q -p_{-k}) \cdot \nu_q |}{|q - p_{-k}|}	
		&
		\geq 
		\left|(e_k + 
		\frac{w_k - w_{-k}}{20d}
		)
		 \cdot  \nu_q
		 \right|
		\geq
		|\mu_k| - \sum_{j=1}^d \frac{1}{10d} |\mu_j|
		\geq 
		\frac{1}{2 \sqrt{d}}
		,
		\llabel{EQ41}
	\end{split}
   \llabel{EQ24}
\end{align}
from where
\begin{align}
	\max_{ p \in \{p_k, p_{-k}\}} \frac{|(q -p) \cdot \nu_q |}{|q - p|} 
	\geq 
	\frac{1}{4 \sqrt{d}}
	,
	\llabel{EQ37}
\end{align}
and
\begin{align}
	\min_{p \in \{p_k, p_{-k}\}} |q-p| \geq 1/10d	
	.
	\llabel{EQ60}
\end{align}
If $q\in B_{1/5d} (e_{k})$, then $q= e_k + w/5d$ for some $w\in B_1$. 
Consequently, we have
\begin{align}
	\begin{split}
		\frac{|(q -p_{-k}) \cdot \nu_q |}{|q - p_{-k}|}	
		\geq
		\frac{1}{2}
		|(e_k + w/5d - e_{-k} - w_{-k}/10d) \cdot \nu_q|
		\geq
		|\mu_k|
		-
		\frac{3}{20d} \sum_{j=1}^d |\mu_j|
		\geq
		\frac{1}{4\sqrt{d}}
		\label{EQ43}
	\end{split}
\end{align}
and $|q-p_{-k}| \geq 1/10d$.
For the case $q\in B_{1/5d} (e_{-k})$, we proceed analogously as in \eqref{EQ43}  and thus complete the proof of the claim.
Thus by continuity, there exists a sufficiently small constant $\epsilon>0$ independent of $q$ and $p$ such that
\begin{align}
	\begin{split}
		\frac{|(q' -p) \cdot \nu_{q} |}{|q' - p|}	
		\geq
		\frac{1}{8\sqrt{d}}
		,
		\llabel{EQ46}
	\end{split}
\end{align}
for any $q' \in B_\epsilon (q) \cap N$.

For each $p\in T$, denote by $F_p$ the collection of points $q\in N$ such that there exists a unit normal vector $\nu_q$ to $N$ at $q$ with \eqref{EQ44} and~\eqref{EQ45}.
Since $N$ has at most $n$ intersections with any line through $p$, we infer that the $(d-1)$-dimensional Hausdorff measure of $F_p$ is bounded by $n$ times the $(d-1)$-dimensional Hausdorff of the sphere $\mathbb{S}^{d-1}_4 (p) =\{x\in \RR^d: |x-p|=4\}$, up to a universal constant depending on~$d$.
Namely, we have
\begin{align}
	\mathcal{H}^{d-1}
	(F_p)
	\leq
	Cn \mathcal{H}^{d-1}
	( \mathbb{S}^{d-1}_4 (p))
	\leq
	Cn
	,
   \llabel{EQ17}
\end{align}
where $C>0$ is a constant. 
Therefore, the proof of \eqref{EQ47} is completed by applying the above inequality to all $p\in T$ and noting that $N \subset \cup_{p\in T} F_p$.
Finally, rescaling \eqref{EQ47} to a ball of radius $2r$ gives the factor $r^{d-1}$ in~\eqref{EQ39}. 
\end{proof}

\begin{Remark}
	\label{R02}
{\rm
Note that the minimal number of points needed in $T$ is~$d+1$. 
Indeed, we use a simple fact that if the determinant of $d$ vectors is nonzero, then these $d$ vectors are non-coplanar.	
Consequently, the angle between $q-p$ and $\nu_q$ in \eqref{EQ44} is bounded away from zero.
}
\end{Remark}

\begin{proof}[Proof of Theorem~\ref{T01}]

Fix $p\in \TT^d$ and $t\in (0,t_0]$.
Let $r\in (0,1/2]$ be a small constant to be determined below.
For each $j\in \{\pm 1, \ldots, \pm d \}$, there exist a point $q_j \in B_{r/10d} (p+ re_j)$ such that 
\begin{align}
	|u(q_j)| 
	\geq 
	\frac{1}{2}
	\Vert u\Vert_{L^\infty (B_{r/10d} (p+ re_j) )},
	\llabel{eq30},
\end{align}
where $e_j$ is the standard basis in~$\RR^d$. Thus, from Lemma~\ref{Lku2} and \eqref{EQ84} it follows that
\begin{align}
	\begin{split}
	\Vert u(t)\Vert_{L^2 (\TT^d)} 	
	&\leq
		C^{
	K^2 \log (\frac{1}{r})\frac{1}{t}}
		\Vert u(t)\Vert_{L^2 (B_{r/10d} (p+ re_j) )} 
		\\&
		\leq
		r^{\frac{d}{2}}
	C^{
		K^2 \log (\frac{1}{r})\frac{1}{t}}
	\Vert u(t)\Vert_{L^\infty (B_{r/10d} (p+ re_j) )} 	
	\leq 
	2	r^{\frac{d}{2}}
C^{
	K^2 \log (\frac{1}{r})\frac{1}{t}}
|u(q_j) |.
\label{EQ40}
	\end{split}
\end{align}
Fix $j\in \{\pm 1, \ldots, \pm d \}$ and a line $l$ through~$q_j$.
We assume that $u$ has $n$ zeros in the line segment $l \cap B_{2r} (p)$, counting the multiplicity. We rotate the coordinate so that $u$ is a function of a single variable and we use $u^{(m)}$ to denote the $m$-th derivative of~$u$.
Let $\{x_1, \ldots, x_k\}$ be $k$ distinct zeros in $l \cap B_{2r} (p)$ with multiplicities $m_1, \ldots, m_k\geq 1$ and $m_1+ \cdots+ m_k = n$.
By the Hermite interpolation theorem, there exists a unique interpolation polynomial $P_{n-1}$ of $u$, of degree less than or equal to $n-1$, satisfying
\begin{align}
	P^{(m)}_{n-1} (x_l)
	= 
	u^{(m)} (x_l)
	\comma
	0 \leq m \leq m_l -1
	\comma
	l=1,\ldots,k.
	\label{EQ15}
\end{align}
Moreover, for any $x\in l \cap B_{2r} (p)$, there exists some $\xi \in l \cap B_{2r} (p)$ depending on $x$ such that
\begin{align}
	u(x) - P_{n-1} (x)
	=
	\frac{(x-x_1)^{m_1} \cdots (x-x_k)^{m_k}}{n!} 
	u^{(n)} (\xi)
	,
   \llabel{EQ25}
\end{align}
from where
\begin{align}
	\begin{split}
	\Vert u- P_{n-1}\Vert_{L^\infty (l \cap B_{2r} (p) )}
	\leq
	\frac{\sup_{x\in l \cap B_{2r} (p)}  |(x-x_1)^{m_1} \cdots (x-x_k)^{m_k}|}{n!} 
	\Vert  u^{(n)} \Vert_{L^\infty (l \cap B_{2r} (p))}	
	.
	\end{split}
   \llabel{EQ27}
\end{align}
From \eqref{EQ15} and the uniqueness it follows that $P_{n-1} \equiv 0$.
Consequently, we get
\begin{align}
	\Vert u\Vert_{L^\infty (l \cap B_{2r} (p))}
	\leq
	\Vert u-P_{n-1}\Vert_{L^\infty (l \cap B_{2r} (p))}
	+
	\Vert P_{n-1}\Vert_{L^\infty (l \cap B_{2r} (p))}
	\leq
	\frac{C^n r^n}{n!} 
	\Vert  u^{(n)} \Vert_{L^\infty (l \cap B_{2r} (p))}
	,
	\label{EQ31}
\end{align}
where $C>0$ is a constant.
From \eqref{EQ40} and \eqref{EQ31} it follows that
\begin{align}
	\begin{split}
	\Vert u(t)\Vert_{L^2 (\TT^d)}
		&
	\leq	
	2	r^{\frac{d}{2}}
	C^{
		K^2 \log (\frac{1}{r})\frac{1}{t}}
\Vert u(t)\Vert_{L^\infty (l \cap B_{2r}(p))}
	\leq
	\frac{C^{n} r^{n+\frac{d}{2}}}{n!}
	C^{ K^2 \log (\frac{1}{r})\frac{1}{t}}
	\Vert  u^{(n)} \Vert_{L^\infty (l \cap B_{2r}(p))}
	\llabel{EQ12}
	,
	\end{split}
\end{align}
since $q_j\in l\cap B_{2r} (p)$.
Appealing to the Sobolev inequality and Lemma~\ref{Lgevrey}, we obtain
\begin{align}
	\begin{split}
		\Vert u(t)\Vert_{L^2 (\TT^d)}
		&
		\leq
		\frac{C^{n} r^{n+\frac{d}{2}}}{n!}
		C^{
			K^2 \log (\frac{1}{r})\frac{1}{t}}
		\Vert (A^d+I) A^{\frac{n}{2}} u(t) \Vert_{L^2 (\TT^d)}		
		\\&
		\leq
		C_0^{n+
		K^2 \log (\frac{1}{r})\frac{1}{t}+
		\delta ^{2/(2-\beta)} t^{-\beta/(2-\beta)} + t(K_w^2 +K_v+M_1  +M_0 +q_0) }
		\\&\indeq\times
		\frac{r^{n+\frac{d}{2}}}{n!}
		 (n+2d)!^{\frac{1}{\beta}}
		\Vert u(t) \Vert_{L^2 (\TT^d )}.
		\label{EQ32}
	\end{split}
\end{align}
for some constant $C_0>0$.
Denote by $C_1:= K_w^2 +K_v  +M_1+M_0 +q_0$ and $C_2:=\delta^{2/(2-\beta)}$. In order for the inequality \eqref{EQ32} to hold, it is necessary that
\begin{align}
	C_0^{n+K^2 \log (\frac{1}{r})\frac{1}{t}+C_1 t +
	C_2 t^{-\beta/(2-\beta)} }
	\frac{ r^{n+\frac{d}{2}}}{n!}
	(n+2d)!^{\frac{1}{\beta}}
	\geq 
	1.
	\label{EQ14}
\end{align}
Note that by Stirling formula there exists a universal constant $C_3>0$ such that $C_3^n n^{1/2+n}<n! \leq n^n$ for all $n\in \mathbb{N}$. 
Taking the logarithm of \eqref{EQ14} and appealing to the Stirling formula, we obtain
\begin{align}
\begin{split}
	&
	\left(n+\frac{K^2}{t} \log \frac{1}{r}
	+
	C_1 t 
	+
	C_2 t^{-\beta/(2-\beta)}
	\right) 
	\log C_0
	\\&\indeq
	+
	\left(n+\frac{d}{2} \right) 
	\log r
	-
	\left(\frac{1}{2}+n \right)  \log n
	-
	n \log C_3
	+
	\frac{n+2d}{\beta} \log (n+2d)
	\geq 
	0.
	\label{EQ13}	
\end{split}
\end{align}

Let $n_0$ be the largest integer less than or equal to $
	2K^2 \log (\frac{1}{r})\frac{1}{t}+2d
	+1$.
In the left side of \eqref{EQ13} we replace $n$ with $n_0$, and the resulting quantity is bounded from above by
\begin{align}
	\begin{split}
	&
 	\left(C_1 t 
 	+
 	C_2 
 	t^{-\frac{\beta}{2-\beta}}
 \right)\log C_0
 	+
 	n_0  
 	\log \left(\frac{C_0^2 r 2^{\frac{1}{\beta}}  n_0^{\frac{1}{\beta}-1} }{C_3} \right)
 	+\frac{2d}{\beta} \log (2n_0)
 	.
 	\label{EQ50}
	\end{split}
\end{align}
Let 
\begin{align}
	r=
	\frac{t^{\frac{1}{\beta} -1}}{\log^{\frac{1}{\beta}-1} (\frac{1}{t})M},
	\label{EQ85}
\end{align}
where $M\geq 1$ is a sufficiently large constant depending on $C_0, C_1, C_2$, $C_3$ and~$K$. Then the quantity in \eqref{EQ50} is less than $0$ for all $t\in (0,t_0]$, since $\beta/(2-\beta)\leq 1$. Consequently, the value $n = n_0$ does not satisfy \eqref{EQ13} and thus does not satisfy~\eqref{EQ32}. We infer that $u$ has less than $n_0$ zeros in $l\cap B_{2r} (p)$ for fixed $r>0$ as in~\eqref{EQ85}.
Assume, for the sake of contradiction, that
	the order of vanishing of $u(x,t)$ is infinite for some $x\in\TT^d$ and some $t\in (0,t_0]$. Then, for the line $l$ passing $q_j$ and $x$,
	the restriction of $u$ to $l\cap B_{2r} (x)$ has a zero at $x$ of infinite multiplicity. 
	This contradicts with the fact that the multiplicity of zeros of $u$ is bounded by $n_0$ in $l\cap B_{2r} (p)$, uniformly for any $p\in \TT^d$.

For $t\in (0,t_0]$, since the order of vanishing for $u(x,t)$ is finite for any $x\in \TT^d$,
we infer from \cite[Claim~2]{M} that $N(t)$ is contained in a union of $(d-1)$-dimensional $C^1$ graphs. Without loss of generality, we may assume that the union is taken on a countable set, as otherwise we consider any countable subset of the uncountable  union and take the supremum.
%
%
From Lemma~\ref{L03} it follows that
\begin{align}
	\mathcal{H}^{d-1}
	(N(t) \cap B_{2r} (p))
	\leq
	C n_0 r^{d-1}
	.
	\llabel{EQ34}
\end{align}
We cover $\TT^d$ using a finite number of balls $B_{2r} (p)$ to get
\begin{align}
	\mathcal{H}^{d-1}
	(N(t))
	\leq 
	C n_0 r^{d-1} \frac{1}{r^{d}}
	\leq
	C 	\left(\frac{1}{t} \right)^{\frac{1}{\beta}} \log^{2 (\frac{1}{\beta}-1)} \left(\frac{1}{t} \right)
	,
	\llabel{EQ65}
\end{align}
where $C>0$ is a sufficiently large constant depending on $q_0$, $M_0$, $M_1$, $K_v$ and~$K_w$.
The proof of \eqref{EQ26} is thus completed.
\end{proof}

\colb
\section*{Acknowledgments}
IK was supported in part by the NSF grant DMS-2205493.

%
%

\end{document}